\documentclass[12pt, a4paper]{amsart}
\usepackage{amscd,amsmath,amssymb,amsthm,amsfonts}
\usepackage[alphabetic]{amsrefs}
\usepackage{mathrsfs}
\usepackage[shortlabels]{enumitem}
\usepackage[all]{xy}

\usepackage{xcolor}
\usepackage[hypertexnames=true, citecolor = green, colorlinks = false, linkcolor = red, linktoc = none, pdffitwindow = false, urlbordercolor = white]{hyperref}%

\def\id{\operatorname{id}}

\def\id{\operatorname{id}}

\def\kms{\operatorname{KMS}}

\def\C{\mathbb{C}}

\def\R{\mathbb{R}}
\def\N{\mathbb{N}}
\def\Z{\mathbb{Z}}
\def\T{\mathbb{T}}

\def\Q{\mathbb{Q}}

 \newcommand{\IF}[0]{\mathbb{F}}

\newcommand{\CA}[0]{\mathcal{A}} 
 \renewcommand{\CD}[0]{\mathcal{D}}
 \newcommand{\CF}[0]{\mathcal{F}}
 
 \newcommand{\CJ}[0]{\mathcal{J}}

\newcommand{\CO}[0]{\mathcal{O}} \newcommand{\CP}[0]{\mathcal{P}}
\newcommand{\CQ}[0]{\mathcal{Q}} 
 \newcommand{\CT}[0]{\mathcal{T}}

\def\gxp{G \rtimes_\theta P}
\def\gpt{G,P,\theta}

\newtheorem{thm}{Theorem}[section]
\newtheorem{corollary}[thm]{Corollary}
\newtheorem{lemma}[thm]{Lemma}

\newtheorem{proposition}[thm]{Proposition}

\theoremstyle{definition}
\newtheorem{definition}[thm]{Definition}

\theoremstyle{remark}
\newtheorem{remark}[thm]{Remark}

\newtheorem{question}[thm]{Question}
\newtheorem{example}[thm]{Example}

\newtheorem{examples}[thm]{Examples}

\numberwithin{equation}{section}

\begin{document}

\title{A boundary quotient diagram for right LCM semigroups}

\author{Nicolai Stammeier}
\address{Department of Mathematics \\ University of Oslo \\ P.O.~Box 1053 Blindern \\ NO-0316 Oslo, Norway}
\email{nicolsta@math.uio.no}
\thanks{The author was supported by ERC through AdG 267079 and by RCN through FRIPRO 240362.}
\subjclass[2010]{46L05 (Primary) 20M30, 46L30 (Secondary)}

\begin{abstract}
We propose a boundary quotient diagram for right LCM semigroups with property (AR) that generalizes the boundary quotient diagram for $\N \rtimes \N^\times$. Our approach focuses on two important subsemigroups: the core subsemigroup and the semigroup of core irreducible elements. The diagram is then employed to unify several case studies on KMS-states, and we end with a discussion on $K$-theoretical aspects of the diagram motivated by recent findings for integral dynamics.
\end{abstract}

\maketitle

\section{Introduction}\label{sec:intro}
\noindent A countable discrete semigroup $S$ is called \emph{right LCM} if it is left cancellative and the intersection of two principal right ideals in $S$ is either empty or another principal right ideal. The terminology alludes to the existence of right least common multiples given the existence of any common right multiple. It is known that a left cancellative semigroup $S$ is right LCM if and only if Li's family of constructible right ideals $\CJ(S)$ is given by $\emptyset$ and the principal right ideals in $S$, see \cites{Li1,BLS1}. We shall assume that the semigroup $S$ is not only right LCM and unital, but also has property (AR), which is explained in Section~\ref{sec:construction of the BQD}. But we remark here that no example of a right LCM semigroup without this property is known so far, see \cite{BS1}.

A classical example of a right LCM semigroup $S$ is $\N \rtimes \N^\times$. Within the operator-algebraic context, this example is treated in great detail in the celebrated work \cite{LR2}, where Laca and Raeburn studied the Toeplitz algebra for the quasi-lattice ordered pair $(\Q \rtimes \Q_+^\times, \N \rtimes \N^\times)$, in particular with regards to its KMS-state structure for a natural dynamics. They also considered the boundary quotient $\CQ(\N \rtimes \N^\times)$ of $\CT(\N \rtimes \N^\times)$ in the sense of Crisp and Laca \cite{CrispLaca}. Using a suitable presentation for $\CT(\N \rtimes \N^\times)$ by generators and relations, they show that $\CQ(\N \rtimes \N^\times)$ is obtained from $\CT(\N \rtimes \N^\times)$ by imposing two extra relations: 
\begin{enumerate}
\item[(a)] The isometry corresponding to $(1,1) \in \N \rtimes \N^\times$ is a unitary.
\item[(m)] For every prime $p \in \N^\times$, the isometries for $\{(m,p) \mid 0 \leq m \leq p-1\}$ form a Cuntz family, that is, their range projections sum up to one.
\end{enumerate}
As a straightforward consequence, $\CQ(\N \rtimes \N^\times)$ coincides with Cuntz's $\CQ_\N$ from \cite{CuntzQ}.


The results of \cite{LR2} were extended in \cite{BaHLR} by introducing and analysing two complementary intermediate quotients between $\CT(\N \rtimes \N^\times)$ and $\CQ(\N \rtimes \N^\times)$: the additive boundary quotient $\CT_{\text{add}}(\N \rtimes \N^\times)$, obtained by imposing (a), and the multiplicative boundary quotient $\CT_{\text{mult}}(\N \rtimes \N^\times)$, obtained by imposing (m), see \cite{BaHLR}*{Proposition~3.3}. 
Altogether, these form the \emph{boundary quotient diagram} for $\N \rtimes \N^\times$:
\begin{equation*}\label{eqref:BQD NxN^times}
\begin{gathered}
\xymatrix@=17mm{
\CT(\N \rtimes \N^\times) \ar@{->>}[r]\ar@{->>}[d] & \CT_{\text{mult}}(\N \rtimes \N^\times) \ar@{->>}[d] \\
\CT_{\text{add}}(\N \rtimes \N^\times) \ar@{->>}[r] & \CQ(\N \rtimes \N^\times)
}
\end{gathered}
\end{equation*}
This diagram was shown to exhibit interesting features with respect to KMS-states, see \cite{BaHLR}*{Section~4}. 

By now, several works on KMS-state structures on Toeplitz type algebras and their quotients have been influenced, if not very much inspired by the approach in \cite{LR2}, see for instance \cites{LRR,LRRW,CaHR}. Somewhat intriguingly, the treatment for Baumslag-Solitar monoids $BS(c,d)^+$ features a boundary quotient diagram in disguise, see \cite{CaHR}*{Corollary~5.3}. For both $\N \rtimes \N^\times$ and $BS(c,d)^+$, the choices of the intermediate quotients are natural, yet based on the particular presentation of the semigroup. In addition, each of the aforementioned accounts on KMS-state structures remains an isolated case study for a specific family of right LCM semigroups, even though the similarities with regards to results and methods of proof are apparent.

One central aim of this work is to overcome this deficiency by introducing a boundary quotient diagram  \eqref{eqref:BQD} for every right LCM semigroup with property (AR) that allows us to display the results on KMS-states from \cites{LR2,BaHLR,LRR,LRRW,CaHR} in a unified manner, see Theorem~\ref{thm:KMS on BQD}. 

A convenient framework for this is provided through Li's theory \cite{Li1} of full semigroup $C^*$-algebras $C^*(S)$ and the notion of a boundary quotient $\CQ(S)$ for right LCM semigroups $S$ from \cite{BRRW}, as these constructions generalize the corresponding ones for quasi-lattice ordered groups. Thus the task reduces to identifying two natural intermediate quotients that complement each other in a suitable sense. To this end, we recall that $\CQ(S)$ is obtained from $C^*(S)$ by imposing the boundary relation $\sum_{f \in F} e_{fS}=1$ for all accurate foundation sets $F$, see Section~\ref{sec:construction of the BQD} for details. The singleton foundation sets play a special role: They are given by the elements of the \emph{core subsemigroup} $S_c = \{s \in S \mid sS \cap tS \neq 0 \text{ for all } t \in S\}$, and the boundary relation turns the corresponding generating isometries $v_s \in C^*(S)$ into unitaries. This will serve as our defining relation for the \emph{core boundary quotient} $\CQ_c(S)$, which generalizes $\CT_{\text{add}}(\N \rtimes \N^\times)$. 

The naive approach of defining the analogue of $\CT_{\text{mult}}(\N \rtimes \N^\times)$ as the quotient of $C^*(S)$ by the boundary relation for every accurate foundation set $F$ with $\lvert F \rvert \geq 2$, or equivalently $F \subset S\setminus S_c$, is bound to fail. Indeed, it is easy to see that we get nothing but $\CQ(S)$ in this case. Therefore, we propose a slightly more elaborate version: We call an element $s \in S \setminus S_c$ \emph{core irreducible} if every factorization $s=tr$ with $r \in S_c$ satisfies $r \in S^*$, where $S^*$ denotes the subgroup of invertible elements in $S$. Since $S$ is assumed left cancellative, the core irreducible elements form a subsemigroup $S_{ci}$ of $S$. We say that a foundation set $F$ is \emph{proper} if $F$ consists of core irreducible elements, and then define the \emph{proper boundary quotient} $\CQ_p(S)$ as the quotient of $C^*(S)$ by the boundary relation for all proper accurate foundation sets. Thus, the conditions
\begin{enumerate}
\item[(c)] For every $s \in S_c$, the isometry $v_s$ is a unitary.
\item[(p)] For every proper accurate foundation set $F$, the Toeplitz-Cuntz family of isometries $(v_f)_{f \in F}$ is a Cuntz family. 
\end{enumerate} 
replace (a) and (m) from $\N \rtimes\N^\times$ for a general right LCM semigroup with property (AR), and thus giving rise to the \emph{boundary quotient diagram}
\[\xymatrix@=17mm{
C^*(S) \ar@{->>}[r]\ar@{->>}[d] & \CQ_p(S) \ar@{->>}[d] \\
\CQ_c(S) \ar@{->>}[r] & \CQ(S)
}\]

We then show that, under mild assumptions, $\CQ_c(S)$ and $\CQ_p(S)$ are complementary quotients in between $C^*(S)$ and $\CQ(S)$ in the sense that (c) and (p) together yield $\CQ(S)$, see Proposition~\ref{prop:ACC gives complementary maps in the BQD}. As a continuation of considerations from \cite{BLS2}, we then describe sufficient conditions under which semigroup homomorphisms between right LCM semigroups give rise to $*$-homomorphisms between corresponding corners of the boundary quotient diagrams, see Remark~\ref{rem:BQD induced maps}. 

In order to demonstrate the utility of our approach, various examples are discussed in Section~\ref{sec:examples}. In particular, our boundary quotient diagram is shown to explain the appearance of the two intermediate quotients $C^*_A(U \bowtie A)$ and $C^*_U(U \bowtie A)$ in \cite{BRRW}*{Remark~5.4} for rather special Zappa-Sz\'{e}p products $U \bowtie A$ of right LCM semigroups, see Example~\ref{ex:Zappa-Szep products}. More importantly, our perspective indicates that the focus on the two components $U$ and $A$ is somewhat misleading: The reason why the two quotients agree with $\CQ_c(U \bowtie A)$ and $\CQ_p(U \bowtie A)$ is that the prescribed conditions $U$ and $A$ need to satisfy force $(U \bowtie A)_c = U^* \bowtie A_c$ and $(U \bowtie A)_{ci} = U_{ci} \bowtie A^*$, so that proper accurate foundation sets of $U \bowtie A$ are essentially determined by proper accurate foundation sets of $U$.

The structure of the subsemigroup of \emph{core irreducible} elements $S_{ci}$ is of independent interest. It appears to contain vital information on the semigroup $S$ itself, see Remark~\ref{rem:S=S_ciS_c and S* trivial -> ZS-product}. With regards to the associated $C^*$-algebras, we provide evidence that $S_{ci}$ plays an important role in the quest for the $K$-theory of $\CQ(S)$ and $\CQ_p(S)$. This is discussed in Section~\ref{sec:k-theory} in the context of integral dynamics \cite{BOS1} and Baumslag-Solitar monoids \cite{Spi1}. 

We suspect the boundary quotient diagram to admit an elegant and equally useful description in the language of groupoids. In particular, this might be the key to studying a boundary quotient diagram for countable discrete left cancellative semigroups and to obtaining a vast generalization of Theorem~\ref{thm:KMS on BQD}. In this direction, the work of Laca and Neshveyev \cite{LaNe} may be crucial, and its appendix indicates the existence of a common theme for KMS-state structures of this kind. Having said that, we will not address this here for the sake of an elementary and brief exposition.

The paper is organized as follows: The boundary quotient diagram for right LCM semigroups with property (AR) is constructed in Section~\ref{sec:construction of the BQD}. The particular form of the boundary quotient diagram for a variety of examples is discussed in Section~\ref{sec:examples}. As an application of the diagram, a unifying statement for the results on KMS-states from the four different case studies \cites{LR2,BaHLR,LRR,LRRW,CaHR} is established in Section~\ref{sec:KMS}. In the final Section~\ref{sec:k-theory}, these results are contrasted by $K$-theoretical considerations, where we review the torsion subalgebra for integral dynamics in order to present two candidates for an analogue of this subalgebra for other right LCM semigroups.

\emph{Acknowledgements}: The author would like to thank Nadia S.~Larsen, Nathan Brownlowe, Sel\c{c}uk Barlak, Dave Robertson, and Magnus Dahler Norling for stimulating discussions. Parts of this research were carried out during a visit to the University of Wollongong, and the author would like to express his gratitude for the great hospitality of its operator algebra group.


\section{Construction of the boundary quotient diagram}\label{sec:construction of the BQD}
In \cite{BRRW}, a boundary quotient $\CQ(S)$ was introduced for right LCM semigroups $S$ as the quotient of the full semigroup $C^*$-algebra $C^*(S)$ by the relation $\prod_{f \in F}(1-e_{fS}) = 0$ for all foundation sets $F$ for $S$. Recall that a finite subset $F$ of $S$ is called a \emph{foundation set}, if, for every $t \in S$, there is $s \in F$ such that $sS \cap tS \neq \emptyset$. For convenience, let us denote the set of all foundation sets for $S$ by $\CF(S)$. This approach to defining $\CQ(S)$ was inspired by the work of Crisp and Laca in the setting of quasi lattice-ordered groups \cite{CrispLaca}. Shortly thereafter, it was observed in \cite{BS1} that a broad class of right LCM semigroups has the \emph{accurate refinement property}, or property (AR) for short: For every $F \in \CF(S)$, there is $F' \in \CF(S)$ such that 
\begin{enumerate}[a)]
\item $F'$ is \emph{accurate} ($fS$ and $f'S$ are disjoint for $f,f' \in F', f\neq f'$), and 
\item $F'$ refines $F$ (for every $f' \in F'$ there is $f \in F$ with $f' \in fS$).
\end{enumerate}
If $S$ has property (AR), then the boundary relation for $\CQ(S)$ reduces to
\begin{equation}\label{eq:sum boundary relation}
\begin{array}{c} \sum\limits_{f \in F}e_{fS} = 1 \end{array}
\end{equation}
for every accurate foundation set $F$, the collection of which we shall denote by $\CF_a(S)$. 


A simple, but important observation from \cite{Star} is the relevance of the \emph{core subsemigroup} 
\[\begin{array}{c} S_c := \{ s \in S \mid sS \cap tS \neq \emptyset \text{ for all } t \in S\}, \end{array}\]
whose origin can again be traced back to \cite{CrispLaca}. The semigroup $S_c$ contains the group of units $S^*$, and forms a right reversible semigroup, that is, finite intersections of nonempty right ideals are nonempty. Hence $S_c$ is a right Ore semigroup provided that it has right cancellation. With regards to the boundary quotient $\CQ(S)$, we note that every generating isometry $v_s \in C^*(S)$ with $s \in S_c$ is turned into a unitary when passing to $\CQ(S)$. This motivates the definition of the first intermediate quotient between $C^*(S)$ and $\CQ(S)$.

\begin{definition}\label{def:core BQ}
The \emph{core boundary quotient} $\CQ_c(S)$ is the quotient of $C^*(S)$ by the relation $v_s^{\phantom{*}}v_s^* = 1$ for all $s \in S_c$.
\end{definition}

One may be tempted to define the second intermediate quotient as the quotient of $C^*(S)$ by \eqref{eq:sum boundary relation} restricted to (accurate) foundation sets that do not contain any element from the core $S_c$. However, this yields nothing but $\CQ(S)$ as we shall now see. The starting point are the following two basic observations whose straightforward proofs are left to the reader.

\begin{lemma}\label{lem:FS concatenation}
For $F_1,F_2 \subset S$, the set $F_1 \cdot F_2 := \{st \mid s \in F_1,t \in F_2\} \in \CF(S)$ is an accurate foundation set if and only if $F_1$ and $F_2$ are accurate foundation sets.
\end{lemma}

If $F_i = \{s\}$ for some $s \in S$, we shall simply write $s \cdot F_2$ or $F_1 \cdot s$, respectively.

\begin{lemma}\label{lem:boundary relation for products of FS}
Let $F_1,F_2 \in \CF_a(S)$. Then the boundary relation \eqref{eq:sum boundary relation} for both $F_1$ and $F_2$ is equivalent to the boundary relation \eqref{eq:sum boundary relation} for $F_1 \cdot F_2$.
\end{lemma}

As a direct consequence of Lemma~\ref{lem:FS concatenation} and Lemma~\ref{lem:boundary relation for products of FS}, we get:

\begin{corollary}\label{cor:non proper FS}
Let $F \in \CF_a(S)$ and $s \in S_c $. Then $s \cdot F,F \cdot s \in \CF_a(S)$, and the boundary relation \eqref{eq:sum boundary relation} for $s \cdot F$ or $F \cdot s$ is equivalent to the boundary relation \eqref{eq:sum boundary relation} for $F$ and $v_s^{\phantom{*}}v_s^* = 1$. 
\end{corollary}

By virtue of Corollary~\ref{cor:non proper FS}, we see that imposing the boundary relation \eqref{eq:sum boundary relation} on $C^*(S)$ for all $F \in \CF_a(S)$ with $F \cap S_c = \emptyset$ still yields $\CQ(S)$. Thus we need a more sophisticated approach, to this end, recall that a non-invertible element $s$ of a monoid $S$ is said to be \emph{irreducible} if $s \notin S^*$ and any decomposition $s=tr$ in $T$ satisfies $t \in S^*$ or $r \in S^*$. 

\begin{definition}\label{def:core irred}
An element $s \in S\setminus S_c$ is called \emph{core irreducible} if $s=tr$ for $t \in S$ and $r \in S_c$ implies $r \in S^*$. The set of core irreducible elements in $S$ is denoted by $S_{ci}$.
\end{definition}

Note that the core irreducible elements are minimal representatives of the equivalence classes in $S/\sim$, where $s \sim t$ if there are $r,r' \in S_c$ such that $sr=tr'$. In addition, let us remark that $S_{ci}$ is a semigroup (without identity) as $S$ is left cancellative, and we denote its unitization by $S_{ci}^{1}$.

\begin{definition}\label{def:proper FS}
A foundation set $F$ for $S$ is called \emph{proper} if $F \subset S_{ci}$. The set of accurate proper foundation sets is denoted by $\CF^{(p)}_a(S)$.
\end{definition}

\begin{definition}\label{def:proper BQ}
The \emph{proper boundary quotient} $\CQ_p(S)$ is the quotient of $C^*(S)$ by the boundary relation \eqref{eq:sum boundary relation} for all proper accurate foundation sets $F$.
\end{definition}

We remark that Definition~\ref{def:proper BQ} does not cater for cases of type $s \cdot F$ with $s \in S_c\setminus S^*$ and $F \in \CF_a^{(p)}(S)$ from Corollary~\ref{cor:non proper FS} explicitly. The reason is that we always get $s \cdot F = F' \cdot s'$ for some $s' \in S_c \setminus S^*$ and $F' \in \CF_a^{(p)}(S)$ in all the examples that we considered, see Section~\ref{sec:examples}. This raises the question whether the definition of $\CQ_p(S)$ ought to be modified:

\begin{question}\label{que:modify Q_p}
Is there a right LCM semigroup $S$ (with property (AR)) for which there are $s \in S_c\setminus S^*$ and $F \in \CF_a^{(p)}(S)$ with $s \cdot F \subset S_{ci}$, i.e.~such that $s \cdot F \in \CF_a^{(p)}(S)$?
\end{question}

With Definition~\ref{def:core BQ} and Definition~\ref{def:proper BQ} at hands, we are ready for the main definition.

\begin{definition}\label{def:BQD}
The \emph{boundary quotient diagram} of a right LCM semigroup $S$ is given by:
\begin{equation}\label{eqref:BQD}
\begin{gathered}
\xymatrix@=17mm{
C^*(S) \ar^{\pi_p}@{->>}[r]\ar_{\pi_c}@{->>}[d] & \CQ_p(S) \ar@{->>}[d] \\
\CQ_c(S) \ar@{->>}[r] & \CQ(S)
}
\end{gathered}
\end{equation}
\end{definition}

It is a natural question whether $\CQ(S)$ can be obtained by imposing the relations for $\CQ_p(S)$ on $\CQ_c(S)$, and vice versa. The next proposition shows that this is indeed the case, given that all elements in $S$ admit a factorization into a core irreducible and a core element. This holds true whenever $S$ satisfies the \emph{ascending chain condition} with respect to $\CJ(S)$, i.e.~every ascending sequence of constructible right ideals becomes stationary. More precisely, this is true if the binary relation $s \to t :\Leftrightarrow s \in t(S_c\setminus S^*)$ is \emph{terminating} as discussed in \cite{Bri1}*{Subsection~2.5}: There is no infinite sequence $(s_n)_{n \geq 1}$ with $s_n \to s_{n+1}, s_n \neq s_{n+1}$ for all $n$. Clearly, if $\to$ is terminating, then $S=S_{ci}^{1}S_c$.

\begin{proposition}\label{prop:ACC gives complementary maps in the BQD}
Let $S$ be a right LCM semigroup with $S=S_{ci}^1S_c$. Then $\CQ(S)$ is the quotient of $\CQ_p(S)$ by the relation $\pi_p(v_s^{\phantom{*}}v_s^*) = 1$ for all $s \in S_c$. Equivalently, $\CQ(S)$ is the quotient of $\CQ_c(S)$ by the relation $\sum_{f \in F} \pi_c(e_{fS}) = 1$ for all accurate proper foundation sets $F$. In particular, this holds true if the relation $\to$ is terminating.
\end{proposition}
\begin{proof}
Let $\CQ'(S)$ be the quotient of $\CQ_p(S)$ obtained by imposing $\pi_p(v_s^{\phantom{*}}v_s^*) = 1$ for all $s \in S_c$. Then $\CQ(S)$ is a quotient of $\CQ'(S)$. Hence it suffices to show that \eqref{eq:sum boundary relation} holds for every accurate foundation set in $\CQ'(S)$. Let $\pi'_c \colon \CQ_p(S) \to \CQ'(S)$ denote the quotient map and suppose $F$ is an accurate foundation set. If $F \cap S_c \neq \emptyset$, then necessarily $F = \{s\}$ for some $s \in S_c$ due to accuracy. But in this case, there is nothing to show. So let $F \subset S\setminus S_c$. By assumption, each $f \in F$ can be written as $f=f_if_c$ with $f_i\in S_{ci}$ and $f_c \in S_c$. Noting that $F_i := \{f_i \mid f \in F\}$ is an accurate proper foundation set, and $\pi'_c(\pi_p(e_{f_iS})) = \pi'_c(\pi_p(e_{fS}))$ as $e_{f_iS}^{\phantom{*}} - e_{fS}^{\phantom{*}}= v_{f_i}^{\phantom{*}}(1-e_{f_cS}^{\phantom{*}})v_{f_i}^*$, we get
\[\begin{array}{c}
\sum\limits_{f \in F} \pi'_c(\pi_p(e_{fS})) = \sum\limits_{f_i \in F_i} \pi'_c(\pi_p(e_{f_iS})) = 1.
\end{array}\]
Thus $\CQ'(S)$ coincides with $\CQ(S)$.
\end{proof}

Given that $S=S_{ci}^1S_c$, Proposition~\ref{prop:ACC gives complementary maps in the BQD} shows that \eqref{eqref:BQD} takes the form
\begin{equation}\label{eqref:BQD with ACC}
\begin{gathered}
\xymatrix@=17mm{
C^*(S) \ar^{\pi_p}@{->>}[r]\ar_{\pi_c}@{->>}[d] & \CQ_p(S) \ar^{\pi'_c}@{->>}[d] \\
\CQ_c(S) \ar^{\pi'_p}@{->>}[r] & \CQ(S)
}
\end{gathered}
\end{equation}
where $\pi'_p$ and $\pi'_c$ are induced by $\pi_p$ and $\pi_c$, respectively. As we shall see in Section~\ref{sec:examples}, all our examples satisfy $S=S_{ci}^1S_c$. This motivates the following two questions:

\begin{question}\label{que:right LCM without S=S_ci S_c}
Is there a right LCM semigroup $S$ that does not satisfy $S=S_{ci}^1S_c$?
\end{question}

\begin{question}\label{que:non-standard BQD}
Is there a right LCM semigroup $S$ for which \eqref{eqref:BQD with ACC} does not hold?
\end{question}

\begin{remark}\label{rem:S=S_ciS_c and S* trivial -> ZS-product}
Let $S$ be right LCM with $S^*=\{1\}$. Then $S=S_{ci}^1S_c$ is precisely what is needed to display $S$ as the internal Zappa-Sz\'{e}p product $S_{ci}^1 \bowtie S_c$, see \cite{Bri1}. If $S^* \neq \{1\}$, but $S_{ci}^1$ admits a transversal $T \subset S_{ci}^1$ for the right action of $S^*$ which forms a semigroup, then $S$ is the internal Zappa-Sz\'{e}p product $T \bowtie S_c$. This is for instance the case for self-similar actions, see Example~\ref{ex:self-similar actions}.
\end{remark}

Let us now examine under which conditions a semigroup homomorphism $\phi\colon S \to T$ between two right LCM semigroups $S$ and $T$ induces maps of the quotients appearing in the boundary quotient diagram \eqref{eqref:BQD}. This constitutes a natural continuation of \cite{BLS2}*{Section~3} leads to new applications, see Section~\ref{sec:examples}. To begin with, observe that $\phi$ is necessarily unital because $T$ is assumed left cancellative and thus the only idempotent in $T$ is $1_T$. Next, recall from \cite{BLS2}*{Theorem~3.3} that $\phi$ induces a $*$-homomorphism $\varphi\colon C^*(S) \to C^*(T)$ if and only if 
\begin{equation}\label{eq:funct C*(S)}
\begin{array}{c} 
\phi(s_1)T \cap \phi(s_2)T = \phi(s_1S \cap s_2S)T \quad \text{for all } s_1,s_2 \in S. 
\end{array}
\end{equation}

\begin{remark}\label{rem:BQD induced maps}
Let $S$ and $T$ be right LCM semigroups. A semigroup homomorphism $\phi\colon S \to T$ satisfying \eqref{eq:funct C*(S)} induces a $*$-homomorphism
\begin{enumerate}[a)]
\item $\varphi_c\colon\CQ_c(S) \to \CQ_c(T)$ if $\phi(S_c)$ is a subsemigroup of $T_c$,
\item $\varphi_p\colon\CQ_p(S) \to \CQ_p(T)$ if $\phi$ maps $\CF^{(p)}_a(S)$ to $\CF^{(p)}_a(T)$, and
\item $\varphi_q\colon\CQ(S) \to \CQ(T)$ if $\phi$ maps $\CF(S)$ to $\CF(T)$.
\end{enumerate}
If \eqref{eqref:BQD} is given by \eqref{eqref:BQD with ACC}, e.g.~if $S=S_{ci}^1S_c$, then condition $\phi(\CF(S)) \subset\CF(T)$ from c) is equivalent to
\begin{enumerate}
\item[c')] $\phi(S_c) \subset T_c$ and $\phi$ maps $\CF^{(p)}_a(S)$ to $\CF^{(p)}_a(T)$.
\end{enumerate}  
\end{remark}

\begin{question}\label{que:necessity of conditions for induced maps}
Are the conditions presented in Remark~\ref{rem:BQD induced maps}~a)--c') necessary?
\end{question}


\section{Examples}\label{sec:examples}
\noindent Within this section we discuss the boundary quotient diagram for a selection of right LCM semigroups encompassing integral dynamics~\ref{ex:NxN^x and subdynamics}, Baumslag-Solitar monoids~\ref{ex:BS(c,d)^+ with cd > 1}, algebraic dynamical systems~\ref{ex:ADS}, Zappa-Sz\'{e}p products~\ref{ex:Zappa-Szep products}, and self-similar actions~\ref{ex:self-similar actions}. In addition, we mention right Ore semigroups in Example~\ref{ex:right Ore}, and indicate obstructions to induced maps for the core boundary quotient for inclusions of right LCM subsemigroups in Example~\ref{ex:induced maps}.

\begin{example}\label{ex:NxN^x and subdynamics}
Let $P \subset \N^\times$ be a monoid generated by a family $\CP$ of relatively prime numbers and consider $S:= \N \rtimes P \subset \N \rtimes \N^\times$. Then $S$ is right LCM, $S^*$ is trivial, and $S_c = \N \times \{1\}$. An element $(n,p) \in S$ is core irreducible if and only if $0 \leq n \leq p-1$, and it is irreducible if, in addition, $p$ is irreducible in $P$. We note that $S^*$ is trivial and $S=S_{ci}^1S_c$ so that $S=S_{ci}^1 \bowtie S_c$, a description that appeared already in \cite{BRRW}*{Subsection~3.2}.

As $P$ is directed, a finite set $F \subset S$ is a foundation set if and only if it is refined by the \emph{elementary foundation set} $F' := \{ (n,p_F) \mid 0 \leq n \leq p_F-1\}$, where $p_F$ is the least common multiple of $\{p \mid (n,p) \in F \text{ for some } n \in \N\}$. A proof of this observation can be obtained along the lines of \cite{BS1}*{Lemma~3.2 and Lemma~3.3}. Note that elementary foundation sets are accurate. As a consequence, it suffices to impose \eqref{eq:sum boundary relation} for proper elementary foundation sets in order to form $\CQ_p(S)$ (and $\CQ(S)$). 

Let $p=p_1\cdots p_n$ be a factorization of $p \in P$ with $p_i \in \CP$ for all $i$. If $F_i$ denotes the elementary foundation set for $p_i$, then each $F_i$ is a proper accurate foundation set and the elementary foundation set for $p$ is given by $F_1\cdots F_n$. By Lemma~\ref{lem:boundary relation for products of FS} and Proposition~\ref{prop:ACC gives complementary maps in the BQD}, we get that:
\begin{enumerate}[(a)]
\item $\CQ_c(S)$ is the quotient by $v_{(1,1)}^{\phantom{*}}v_{(1,1)}^*=1$.
\item $\CQ_p(S)$ is the quotient by $\sum_{0 \leq k \leq p-1} v_{(k,p)}^{\phantom{*}}v_{(k,p)}^* = 1$ for all $p \in \CP$.
\item $\CQ(S)$ is the quotient by $v_{(1,1)}^{\phantom{*}}v_{(1,1)}^*=1$ and $\sum_{0 \leq k \leq p-1} v_{(k,p)}^{\phantom{*}}v_{(k,p)}^* = 1$ for all $p \in \CP$.
\end{enumerate}
In particular, Defintion~\ref{def:BQD} recovers the boundary quotient of \cite{BaHLR} in the case where $\CP$ is the set of all primes, see \cite{BaHLR}*{Proposition~3.3}.
\end{example}

\begin{example}\label{ex:BS(c,d)^+ with cd > 1}
Consider the Baumslag-Solitar monoid $S=BS(c,d)^+ := \langle a,b \mid ab^c=b^da \rangle$ for $c,d \in \N^\times$ with $cd > 1$. According to \cite{Spi1}*{Theorem~2.11}, $S$ is quasi-lattice ordered, hence right LCM. By \cite{Spi1}*{Proposition~2.3}, every $s \in S$ admits a unique normal form $s=w_1w_2\cdots w_mb^i$ with $w_k \in F_d := \{ b^\ell a \mid 0 \leq \ell \leq d-1\}$ and $i \in \N$. In particular, $s \mapsto m$ gives a homomorphism $\ell\colon S \to \N$ and we call $\ell(s)$ the \emph{length} of $s$. 

Next we observe that for $s=v_1v_2\cdots v_mb^i, t=w_1w_2\cdots w_nb^j \in S$ with $m \leq n$, the intersection of $sS$ and $tS$ is non-empty if only if $v_k = w_k$ for all $1 \leq k \leq m$, in which case $sS \cap tS = tb^\ell S$ for a suitable $\ell \in \N$. Thus $S_c = \langle b \rangle \cong \N$ and $s=w_1w_2\cdots w_mb^i$ is core irreducible if and only if $i=0$. The element $s$ is irreducible if and only if $m=1$. Thus we see that $S_{ci}^1 = \langle F_d \rangle \cong \IF_d^+$, the free monoid in $d$ generators. In particular, every accurate proper foundation set $F$ for $S$ is of the form $F= (F_d)^k$ for some $k \geq 1$, i.e.~all words in $F_d$ of length $k$. As for Example~\ref{ex:NxN^x and subdynamics}, $S^*$ is trivial, $S=S_{ci}^1S_c$, and thus $S=S_{ci}^1 \bowtie S_c$. Thus by Proposition~\ref{prop:ACC gives complementary maps in the BQD}, the boundary quotient diagram is characterized as follows:
\begin{enumerate}[(a)]
\item $\CQ_c(S)$ is the quotient by $v_b^{\phantom{*}}v_b^*=1$.
\item $\CQ_p(S)$ is the quotient by $\sum_{0 \leq k \leq d-1} v_{b^ka}^{\phantom{*}}v_{b^ka}^* = 1$.
\item $\CQ(S)$ is the quotient by $v_b^{\phantom{*}}v_b^*=1$ and $\sum_{0 \leq k \leq d-1} v_{b^ka}^{\phantom{*}}v_{b^ka}^* = 1$.
\end{enumerate}
Implicitly, $\CQ_c(S)$ and $\CQ_p(S)$ have already appeared in \cite{CaHR}*{Corollary~5.3(a) and (c)}. 
\end{example}

\begin{example}\label{ex:ADS}
For an algebraic dynamical system $(\gpt)$, that is, a right LCM semigroup $P$ acting on a discrete group $G$ by injective group endomorphisms $\theta_p$ so that $pP \cap qP = rP$ forces $\theta_p(G) \cap \theta_q(G) = \theta_r(G)$, we consider the right LCM semigroup $S= \gxp$, see \cite{BLS2} for details. Let $N_p := [G:\theta_p(G)]$ for $p \in P$. We will assume that $N_p=1$ implies $p \in P^*$, and that $P$ is directed with respect to $p \geq q :\Leftrightarrow p \in qP$. Then $S_c = S^* = \gxp^*$ since for every $(g,p) \in S$ with $p \notin P^*$, we have $N_p \geq 2$, so there is $h \in G$ with $h^{-1}g \notin \theta_p(G)$ and hence $(g,p)S \cap (h,p)S = \emptyset$. Therefore, $C^*(S) = \CQ_c(S)$ and every element in $S\setminus S^*$ is core irreducible.

Let $P^{fin} := \{ p \in P \mid N_p < \infty\}$ denote the subsemigroup of \emph{finite} elements in $P$. By \cite{BS1}*{Proposition~3.9}, every foundation set for $S$ can be refined by an elementary foundation set in the sense of \cite{BS1}*{Definition~3.7} provided that every foundation set $F$ for $P$ with $F \subset P^{fin}$ admits an accurate refinement $F'$ that also satisfies $F' \subset P^{fin}$. In particular, $S$ has property (AR) in this case. Let us assume that this holds true. Then a proper elementary foundation set is of the form $F_p := \{(g_1,p),(g_2,p),\ldots,(g_{N_p},p)\}$, where $2 \leq N_p < \infty$ and $\{g_1,\ldots,g_{N_p}\}$ forms a transversal for $G/\theta_p(G)$. If $p = p_1\cdots p_n$ is a factorization into irrecudible elements $p_i \in P$, then $F_{p_1}\cdots F_{p_n}$ forms an elementary foundation set for $p$, and each $F_{p_i}$ is a proper elementary foundation set. Thus, if every $p \in P$ admits a factorization into a product of finitely many irreducible elements, then it suffices to consider elementary foundation sets of irreducibles in $P$, thanks to Lemma~\ref{lem:boundary relation for products of FS}, and therefore $\CQ_p(S) \cong \CQ(S)$ is the quotient of $C^*(S)$ by $\sum_{\overline{g} \in G/\theta_p(G)} v_{(g,p)}^{\phantom{*}}v_{(g,p)}^* = 1$ for all irreducible $p \in P$.
\end{example}

\begin{example}\label{ex:Zappa-Szep products}
Let $S= U \bowtie A$ be a Zappa-Sz\'{e}p product with $U$ and $A$ right LCM semigroups, $\CJ(A)$ totally ordered, and $U \to U, u \mapsto a\cdot u$ bijective for all $a \in A$. According to \cite{BRRW}*{Lemma~3.3}, $S$ is a right LCM semigroup. Then \cite{BRRW}*{Remark~3.4} gives $S^* = U^* \bowtie A^*$ and $(u,a)S \cap (v,b)S \neq \emptyset \Longleftrightarrow uU \cap vU \neq \emptyset$, from which we deduce:
\begin{enumerate}[(a)]
\item The natural inclusions $\phi_U\colon U \to S$ and $\phi_A\colon A \to S$ satisfy \eqref{eq:funct C*(S)}.
\item $S_c = U_c \times A$, and hence $\phi_U(U_c),\phi_A(A) \subset S_c$.
\item A finite set $F \subset S$ is an accurate proper foundation set if and only if $\{u \mid (u,a) \in F \text{ for some } a \in A\}$ is an accurate proper foundation set for $U$. In particular, $\phi$ maps $\CF^{(p)}_a(U)$ to $\CF^{(p)}_a(S)$. We have $A=A_c$ because $\CJ(A)$ is totally ordered, so $A$ does not have any proper foundation sets.
\end{enumerate}
Hence we get a commutative diagram
\begin{equation}\label{eqref:BQD for Zappa-Szep}
\begin{gathered}
\xymatrix@C=5mm@R=5mm{
 C^*(U) \ar@{->>}[dr]\ar@{->>}[dd] \ar[rrrr] &&&& C^*(S) \ar@{->>}[dr]\ar@{->>}[dd] &&&& C^*(A) \ar@{=}[dr] \ar@{->>}[dd] \ar[llll] \\
& \CQ_p(U) \ar@{->>}[dd]\ar[rrrr] &&&& \CQ_p(S) \ar@{->>}[dd] &&&& C^*(A) \ar@{->>}[dd] \ar[llll] \\
\CQ_c(U) \ar@{->>}[dr]\ar[rrrr] &&&& \CQ_c(S) \ar@{->>}[dr] &&&& \CQ(A) \ar@{=}[dr] \ar[llll] \\
& \CQ(U)\ar[rrrr] &&&& \CQ(S) &&&& \CQ(A) \ar[llll]
}
\end{gathered}
\end{equation}
Moreover, we also have $S_{ci} = U_{ci} \bowtie A^*$ and $S=S_{ci}^1S_c$ if and only if $U=U_{ci}^1U_c$ because $A^*U_c = U_cA^*$. In particular, we recover \cite{BRRW}*{Theorem~5.2} and get a conceptual approach to the intermediate quotients $C^*_A(U \bowtie A)$ and $C^*_U(U \bowtie A)$ from \cite{BRRW}*{Remark~5.4}. Note that the quotient $\CQ_c(U \bowtie A)$ is likely to be different from $C^*_A(U \bowtie A)$ as soon as $U_c \neq U^*$. More importantly, our approach provides candidates for intermediate quotients in the case where $A_c$ is a proper subsemigroup of $A$, i.e.~outside the realm of \cite{BRRW}*{Lemma~3.3}.
\end{example}

\begin{example}\label{ex:self-similar actions}
Let $(G,X)$ be a self-similar action and denote by $X^*$ the free monoid in the alphabet $X$. Then $S= X^* \bowtie G$ is a right LCM semigroup which fits into the setup of Example~\ref{ex:Zappa-Szep products} according to \cite{BRRW}*{Theorem~3.8}, which is in fact a result due to Lawson, see \cite{Law1}*{Proposition~3.5 and 3.6}. As $S_c=S^*=G, C^*(G) = \CQ(G), \CQ_c(X^*)=C^*(X^*)$, and $\CQ_p(X^*) = \CQ(X^*) \cong \CO_{\lvert X \rvert}$, the diagram \eqref{eqref:BQD for Zappa-Szep} simplifies to:
\begin{equation}\label{eqref:BQD for SSA}
\begin{gathered}
\xymatrix@C=5mm@R=5mm{
C^*(X^*) \ar@{->>}[dr]\ar[rrrr] &&&& C^*(S) \ar@{->>}[dr] &&&& C^*(G) \ar@{=}[dr] \ar[llll] \\
& \CO_{\lvert X \rvert}\ar[rrrr] &&&& \CQ(S) &&&& C^*(G) \ar[llll]
}
\end{gathered}
\end{equation}
\end{example}

\begin{example}\label{ex:right Ore}
Suppose $S$ is a right LCM semigroup that satisfies the right Ore condition, that is, there is an embedding of $S$ into a group $G$ such that $G=SS^{-1}$. Thanks to well-known results of Ore and Dubreil, the right Ore condition is equivalent to cancellation and left reversibility ($sS \cap tS \neq \emptyset$ for all $s,t \in S$). Under this assumption, we get $S_c = S$, and thus $C^*(S) = \CQ_p(S)$ as well as $\CQ_c(S) = \CQ(S) \cong C^*(G)$.
\end{example}

\begin{examples}\label{ex:induced maps}
Choose a right LCM subsemigroup $S$ of a right LCM semigroup $T$ and let $\phi\colon S \to T$ be the natural inclusion. If the equation \eqref{eq:funct C*(S)} is satisfied by $\phi$, then $S \cap T_c$ is a subsemigroup of $S_c$. If $S \cap T_c$ coincides with $S_c$, then $\phi$ induces a map $\varphi_c\colon \CQ_c(S) \to \CQ_c(T)$, see Remark~\ref{rem:BQD induced maps}. Note that $S \cap T_c$ can be a proper subsemigroup of $S_c$, e.g.~if $S$ is abelian but not contained in $T_c$. For instance, take $T$ to be the free monoid in two generators $a$ and $b$, and let $S$ the free abelian submonoid generated by $a$.
\end{examples}

\section{Towards a unified treatment for KMS states}\label{sec:KMS}
In this section, we use the boundary quotient diagram \eqref{eqref:BQD} to recast the essential results concerning KMS-states 
\begin{enumerate}[(a)]
\item for $\N \rtimes \N^\times$ from \cites{LR2,BaHLR},
\item for $\Z^d \rtimes_A \N$ with $A \in M_d(\Z), \lvert\det A\rvert > 1$ from \cite{LRR},
\item for $X^*\bowtie G$, where $(G,X)$ is a self-similar action with $\lvert X\rvert > 1$, from \cite{LRRW}, and
\item for Baumslag-Solitar monoids $BS(c,d)^+$ with $c,d \in \N^\times, d > 1$ from \cite{CaHR}.
\end{enumerate}

Note that we appeal to the dual picture of $\T^d \rtimes_A \N$ in (b) as opposed to the original treatment in \cite{LRR}. We remark that if $A$ is invertible in $M_d(\Z)$ in case (b), $X$ is a singleton in case (c), or $d =1$ in case (d), the study of KMS-states on the boundary quotient diagram essentially reduces to the study of traces on group $C^*$-algebras, compare \cite{LRR}. These cases will be excluded from our considerations as we intend to focus on proper semigroups.

In all the cases (a)-(d), the semigroup $S$ features a natural homomorphism $N\colon S \to \N^\times, s \mapsto N_s$ arising from a scaling factor $\kappa \in \R_{>0}$ and a length function $\ell \colon S \to \N$ as $N_s := \kappa^{\ell(s)}$. The map $N$ satisfies $N^{-1}(1) = S_c$ and yields a natural dynamics $\sigma$ of $\R$ on $C^*(S)$ via $\sigma_x(v_s) := N_s^{ix} v_s$ for $x \in \R$. 

Define the \emph{$\zeta$-function} for $S$ to be the formal series $\zeta_S(\beta) := \sum_{\overline{s} \in S/S_c} N_s^{-\beta}$ for $\beta \in \R$. Note that $\zeta_S$ converges for all $\beta$ above a so-called \emph{critical inverse temperature} $\beta_c$. The key to uniqueness of $\kms_\beta$-states on $C^*(S)$ for $\beta$ within a critical interval $[1,\beta_c]$ is a notion of minimality for the semigroups in (a)-(d). This data is displayed in the following table:

\[\begin{array}{|c|c|c|c|c|c|}
\hline
\text{type} & S_c & \kappa & \ell\colon S \to \N & \beta_c & \text{minimality}\\
\hline
(a) & \N & 1 & (m,p) \mapsto \log p & 2 & \bigcap_{p \in \N^\times} p\N = \{0\} \\
\hline
(b) & \Z^d & \lvert\det A\rvert & (g,n) \mapsto n & 1 &\bigcap_{n \in \N} A^n(\Z^d) = \{0\} \\
\hline
(c) & G & \lvert X \rvert & (w,g) \mapsto \ell'(w) & 1 & \forall g \in G: \lvert\{g|_w \mid w \in X^*\}\vert < \infty  \\
\hline
(d) & \N & d & \text{length from \ref{ex:BS(c,d)^+ with cd > 1}} & 1 & \bigcap_{n \in \N} \left(\frac{d}{c}\right)^n \cdot \N = \{0\} \\
\hline
\end{array}\]
Apparently, $\N \rtimes \N$ is minimal. As opposed to \cite{LRR}, we do not require that $A$ is a \emph{dilation matrix}, that is, all its eigenvalues need to be larger than one in absolute value, because this feature is only used in \cite{LRR}*{Lemma~5.7}, where the condition $\bigcap_{n \in \N} A^n(\Z^d) = \{0\}$ is then established. The point is that the latter condition is much more natural for the dynamical system $A:\N \curvearrowright \Z^d$ as it expresses minimality of the dual system.  The minimality condition for (d) is equivalent to $c \notin d\N$. 

To distinguish between elements in $C^*(S)$ and $C^*(S_c)$, let the standard generating isometries for $C^*(S_c)$ be denoted by $w_s$.

\begin{thm}\label{thm:KMS on BQD}
Suppose $S$ is a right LCM semigroup of type (a),(b),(c) or (d). Then the KMS-state structure on $C^*(S)$ with respect to the dynamics $\sigma$ given by $\sigma_x(v_s) := N_s^{ix} v_s$ is characterized by:
\begin{enumerate}[(i)]
\item There are no $\kms_\beta$-states for $\beta < 1$.
\item For $\beta \in [1,\beta_c]$, there is a $\kms_\beta$-state $\psi_\beta$ given by $\psi_\beta(v_s^{\phantom{*}}v_t^*) = \delta_{s t} N_s^{-\beta}$ for all $s,t \in S$. If $S$ is minimal, then $\psi_\beta$ is the only $\kms_\beta$-state.
\item For $\beta \in (\beta_c,\infty)$, there is an affine homeomorphism $\tau \mapsto \psi_{\beta,\tau}$ between the tracial states on $C^*(S_c)$ and the $\kms_\beta$-states given by
\[\psi_{\beta,\tau}(v_s^{\phantom{*}}v_t^*) = 
\begin{cases}
N_s^{-\beta} \ \tau(w_y^{\phantom{*}}w_x^*) &\text{if $sS \cap tS = sxS$, $sx=ty$ with $x,y \in S_c$},\\
0 &\text{otherwise.} 
\end{cases}\]
\item There is a one-to-one correspondence $\phi \mapsto \psi_\phi$ between states on $C^*(S_c)$ and ground states on $C^*(S)$ given by $\psi_{\phi}(v_s^{\phantom{*}}v_t^*) = \chi_{S_c}(s)\chi_{S_c}(t) \ \phi(w_s^{\phantom{*}}w_t^*)$. A ground state is a $\kms_\infty$-state if and only if it comes from a tracial state on $C^*(S_c)$.
\end{enumerate}
With regards to the boundary quotient diagram \eqref{eqref:BQD} the following statements hold:
\begin{enumerate}
\item[(v)] All the $\kms_\beta$-states for $\beta \in [1,\infty)$ factor through $\pi_c$.
\item[(vi)] A ground state factors through $\pi_c$ if and only if it is a $\kms_\infty$-state. In particular, every ground state on $\CQ_c(S)$ is a $\kms_\infty$-state.
\item[(vii)] If a $\kms_\beta$-state factors through $\pi_p$, then $\beta=1$. In particular, $\CQ_p(S)$ and $\CQ(S)$ have a unique $\kms_\beta$-state corresponding to $\psi_1$ if $S$ is minimal.
\item[(viii)] There are no ground states on $\CQ_p(S)$, and hence none on $\CQ(S)$.
\end{enumerate}
\end{thm}
\begin{proof}
Using the descriptions obtained in Example~\ref{ex:NxN^x and subdynamics}, Example~\ref{ex:ADS}, Example~\ref{ex:self-similar actions}, and Example~\ref{ex:BS(c,d)^+ with cd > 1} for (a)-(d), we embark on a reference chase, and leave it as an exercise to identify the semigroup $C^*$-algebra $C^*(S)$ with the $C^*$-algebra of Toeplitz type considered in the respective reference. We will prove part (viii) simultaneously for all cases at the end.

For $S=\N \rtimes \N^\times$, the claims follow from \cite{LR2}*{Theorem~7.1} and \cite{BaHLR}*{Section~4}. For $S= \Z^d \rtimes_A \N$, we note that $\pi_c=\id$ and hence $\pi=\pi_p$, so that \cite{LRR}*{Theorem~1.1} yields (i)-(vii).

Next, let $S=X^* \bowtie G$. Part (i) is \cite{LRRW}*{Proposition~4.1(1)}, and (iii) is \cite{LRRW}*{Theorem~6.1}. As $C^*(S) = \CQ_c(S)$, claim (v) is trivial. Also, (iv) and (vi) merge to a single statement that corresponds to \cite{LRRW}*{Proposition~5.3}: Ground states on $C^*(S)$ are $\kms_\infty$-states, and they are given by tracial states on $C^*(S_c) \cong C^*(G)$. The claims (ii) and (vii) are proven in \cite{LRRW}*{Proposition~7.1 and Theorem~7.3}.

Now let $S= BS(c,d)^+$. Then \cite{CaHR}*{Corollary~5.3(a) and (b)} gives (i) and (v). The combination of \cite{CaHR}*{Corollary~5.3(c)} with \cite{CaHR}*{Proposition~7.1} yields (ii) and (vii). Statement (iv) corresponds to \cite{CaHR}*{Theorem~8.1}, and (vi) is an immediate consequence of this. Claim (iii) is essentially provided by \cite{CaHR}*{Theorem~6.1} except for the fact that the authors assume minimality of $S$ in order to show injectivity of the parametrization $\tau \mapsto \psi_{\beta,\tau}$. Thus we need to strengthen this result slightly. To this end, we observe that the formula \cite{CaHR}*{(6.1)} becomes
\[\begin{array}{lcl}
 \zeta_S(\beta)\psi_{\beta,\tau}(v_{b^n}) &=&\tau(v_{b^n}) + \sum\limits_{\substack{k \geq 1: \ n \ \in \ d(d/c)^j\N \\ \text{for } 0 \leq j \leq k-1}} d^{k(1-\beta)} \tau(v_{b^{n(c/d)^k}}) \vspace*{2mm}\\
&=& \tau(v_{b^n}) + \chi_{d\N}(n) \ d^{1-\beta}\sum\limits_{\substack{k \geq 0: \ (c/d)n \ \in \ d(d/c)^j\N \\ \text{for } 0 \leq j \leq k-1}} d^{k(1-\beta)} \tau(v_{b^{n(c/d)^{k+1}}}) \vspace*{2mm}\\
&=& \tau(v_{b^n}) + \chi_{d\N}(n) \ d^{1-\beta}\psi_{\beta,\tau}(v_{b^{n(c/d)}})
\end{array}\]
for every $n \in \N$ within our notation. The analogous formula holds for $v_{b^n}^*$. Hence, if $\tau$ and $\rho$ are tracial states on $C^*(\N)$ with $\psi_{\beta,\tau}=\psi_{\beta,\rho}$, then $\tau=\rho$, without assuming that $S$ is minimal, as in \cites{LRR,LRRW}. This completes (iii).

In all cases, (viii) is a consequence of (iv) and the existence of accurate proper foundation sets, as indicated in \cite{BaHLR}*{End of Section~4}: Since $\pi_p$ is a $*$-homomorphism, every ground state $\phi$ on $\CQ_p(S)$ lifts to a ground state on $C^*(S)$. We then conclude by (iv) that $\phi$ vanishes on all projections $\pi_p(e_{sS})$ with $s \in S \setminus S_c$. So if there is an accurate proper foundation set $F$, then $1 = \phi(1) = \phi(\sum_{f \in F} \pi_p(e_{fS})) = 0$. Thus there are no ground states on $\CQ_p(S)$ and $\CQ(S)$. 
\end{proof}

Assuming $S_c$ to be right cancellative and hence right Ore, let $G_c$ denote the group $S_c^{\phantom{1}}S_c^{-1}$. As the traces on $C^*(S_c)$ correspond to traces on $C^*(G_c)$, it may seem more natural to use the group $C^*$-algebra in Theorem~\ref{thm:KMS on BQD}. However, we emphasize the core subsemigroup $S_c$ here because we like to think of the $\kms_\beta$-states as arising from the $*$-homomorphism $\varphi\colon C^*(S_c) \to C^*(S)$ induced by the inclusion $S_c \subset S$.

The parallels between the results in Theorem~\ref{thm:KMS on BQD} for the different types (a)-(d) extend beyond their mere statements, for which the boundary quotient diagram provides a unifying framework. Indeed, there are striking analogies in the method of proof. It thus seems natural to ask whether there is a unified treatment for KMS-state structures on the proposed boundary quotient diagram \eqref{eqref:BQD} for general right LCM semigroups.

In an attempt to promote this perspective, the author is currently investigating how one may characterize the KMS-state structure for the boundary quotient diagram associated to certain algebraic dynamical systems $(\gpt)$, without using the model $S = \gxp$ explicitly \cite{ABLS}.

\section{A first look at the K-theory}\label{sec:k-theory}
In Section~\ref{sec:KMS} we learned that the left column of the boundary quotient diagram \eqref{eqref:BQD} typically has a significantly richer supply of KMS-states than the right column, see Theorem~\ref{thm:KMS on BQD}. With regards to K-theory however, the situation is almost the opposite: 

\begin{thm}[\cite{BLS2}*{Theorem~5.3}]\label{thm:BLS2 result}
Suppose $S$ is a left Ore right LCM semigroup. If the group $S^{-1}S$ satisfies the Baum-Connes conjecture with coefficients in commutative $C^*$-algebras and the left regular representation $\lambda:C^*(S) \to C_r^*(S)$ is an isomorphism, then $K_*(C^*(S)) \cong K_*(C^*(S^*))$.
\end{thm}

In particular, Theorem~\ref{thm:BLS2 result} applies if $S^{-1}S$ is amenable. For instance,  $S= \N \rtimes P$ as in Example~\ref{ex:NxN^x and subdynamics} satisfies all these conditions. Moreover, $\CQ_c(\N \rtimes P) \cong C^*(\Z \rtimes P)$, and $\Z \rtimes P$ also satisfies the prerequisites of Theorem~\ref{thm:BLS2 result}. Thus we get $K_*(C^*(S)) \cong K_*(\C)$ and $K_*(\CQ_c(S)) \cong K_*(C^*(\Z))$ in these cases.

On the other hand, the computation of the K-theory for $\CQ(S)$  can be a challenging task, already for the case where the semigroup arises from a singly generated dynamical system, see \cite{CuntzVershik}. However, this is perhaps the most interesting case as $\CQ(S)$ happens to a unital UCT Kirchberg algebra for many right LCM semigroups under mild assumptions, see \cites{Star,BS1}. 

For semigroups related to dynamical systems with higher complexity, very few is known. Therefore we will start with a very basic example, namely subdynamics of $\N \rtimes \N^\times$, i.e.~$S= \N \rtimes P$, where $P$ is the free abelian submonoid of $\N^\times$ generated by a family $\CP$ of relatively prime numbers. For such semigroups, nontrivial partial results emerged in the recent past \cites{LN2,BOS1}. These results lead to intriguing questions and relate to conjectures about $C^*$-algebras of $k$-graphs, see \cite{BOS1}*{Conjecture~5.11}.

\begin{remark}\label{rem:K-theory NxN^times}
In \cite{LN2}*{6.3}, Li and Norling determined $K_*(\CQ_p(S))$ for $\lvert\CP\rvert \leq 2$ as well as $K_1(\CQ_p(S))$ for $\lvert\CP\rvert=3$. In joint work with Barlak and Omland \cite{BOS1}, we obtained similar formulas for $K_*(\CQ(S))$ in the case of $\lvert\CP\rvert \leq 2$ or $g_\CP = 1$, where $g_\CP$ is the greatest common divisor of $\CP-1 \subset \N^\times$, i.e.~ $g_\CP := \gcd(\{p-1 \mid p \in \CP\})$. In addition, we proved a number of structural results about $K_*(\CQ(S))$, which we deem relevant here:
\begin{enumerate}[a)]
\item $K_i(\CQ(S)) \cong \Z^{2^{\lvert\CP\rvert-1}} \oplus K_i(\CA(S))$ for $i=0,1$, where  $\CA(S)$ is the subalgebra of $\CQ(S)$ generated by $\{\pi(v_{(n,p)}) \mid p \in \CP, 0 \leq n \leq p-1\}$.
\item $\CA(S) \cong M_{d^\infty}(\C) \rtimes P$ with $d= \prod_{p \in \CP} p$, and hence $\CA(S)$ is a unital UCT Kirchberg algebra. In particular, $\CA(S)$ also embeds into $\CQ_p(S)$ due to simplicity. Thus there is a commutative diagram 
\[\xymatrix@=17mm{
\CQ_p(S) \ar^(.3){\cong}[r]& \varinjlim M_{d^n}(C^*(\N)) \rtimes P& \varinjlim M_{d^n}(C^*(\N)) \ar@{_(->}[l] \\
\CA(S) \ar^(.3){\cong}[r] \ar@{^(->}[u]\ar@{_(->}[d]& M_{d^\infty}(\C) \rtimes P \ar@{^(->}[u]\ar@{_(->}[d] & M_{d^\infty}(\C) \ar@{_(->}[l] \ar@{^(->}[u] \ar@{_(->}[d] \\
\CQ(S) \ar^(.3){\cong}[r]& \varinjlim M_{d^n}(C^*(\Z)) \rtimes P& \varinjlim M_{d^n}(C^*(\Z)) \ar@{_(->}[l]
}\]
given by the natural inclusions of the $P$-invariant subalgebra $M_{d^\infty}(\C)$ into the Bunce-Deddens algebra of type $d^\infty$ and its extension $\varinjlim M_{d^n}(C^*(\N))$. 
\item $S_{ci}^1 = \{(n,p) \in S \mid 0 \leq n \leq p-1\}$ is a right LCM semigroup, $S= S_{ci}^1 \bowtie S_c$, and the inclusion $S_{ci}^1 \subset S$ induces an isomorphism $\CQ(S_{ci}^1) \to \CA(S)$.
\item $\CA(S)$ is isomorphic to $\bigotimes_{p \in \CP} \CO_p$ for $\lvert\CP\rvert \leq 2$. For $\lvert\CP\rvert \geq 3$, the order of every element in $K_*(\CA(S))$ is a divisor of $g_\CP$. This leads to the conjecture that $\CA(S)$ is isomorphic to $\bigotimes_{p \in \CP} \CO_p$ for all families $\CP$, see \cite{BOS1}*{Conjecture~6.5}.
\item By b), $\CA(S)$ embeds canonically into $\CQ_p(S)$ and this embedding yields an isomorphism in $K$-theory, at least for $\lvert\CP\rvert \leq 2$, see \cites{LN2,BOS1}.
\end{enumerate}
\end{remark} 

\noindent The relevance of $\CA(S)$ for the K-theory of $\CQ(S)$ and $\CQ_p(S)$ is apparent.  This shall serve as our motivation for investigating the potentials of two different options for defining $\CA(S)$ in a broader context:\vspace*{2mm} 

\noindent Let $S = \gxp$ for an algebraic dynamical system $(\gpt)$ and denote the diagonal subalgebra in $\CQ(S)$ by $\CD$. Using the approach from \cite{Sta1}, one can show that $\CQ(S) \cong \CD \rtimes S \cong (\CD \rtimes G)\rtimes P$. If, moreover, $P$ is directed with respect to $p \geq q :\Leftrightarrow p \in qP$ and $G/\theta_p(G)$ is finite for all $p \in P$, then we get a generalized Bunce-Deddens algebra $\CD \rtimes G = \lim\limits_{\substack{\longrightarrow \\ p \in P}} M_p(C^*(G))$, and its canonical AF-subalgebra $\lim\limits_{\substack{\longrightarrow \\ p \in P}} M_p(\C)$ is invariant under the $P$-action. In this case, define $\CA(S)$ to be the resulting semigroup crossed, as this is a natural analogue for the torsion subalgebra $\CA(S)$ from Remark~\ref{rem:K-theory NxN^times}.

If $P$ is free abelian, a minor modification of \cite{BOS1}*{Proposition~4.6} then shows that $K_*(\CA(S))$ is again a torsion group, which is finite if $P$ is finitely generated. 

Next, suppose that $\CQ(S)$ to be a unital UCT Kirchberg algebra (or just simple). Then the corresponding property will pass to $\CA(S)$, basically by using the arguments from \cite{BOS1}*{Proposition~5.1 and Corollary~5.2}. In particular, $\CA(S)$ can be identified with a subalgebra of $\CQ(S)$ in this case.\vspace*{2mm}

\noindent An obvious alternative is to consider $\CQ(S_{ci}^1)$ for an arbitrary right LCM semigroup $S$. This is expected to be of greater interest if $S^* = \{1\}$ and $S=S_{ci}^1S_c$. Recall that this allows us to write $S= S_{ci}^1 \bowtie S_c$, see Remark~\ref{rem:S=S_ciS_c and S* trivial -> ZS-product}. As an example, let us look at the outcome for Baumslag-Solitar monoids $S=BS(c,d)^+$ with $c,d \geq 1$: Thanks to the efforts of Spielberg, the $K$-theory for $\CQ(S)$ is already known:

\begin{thm}[\cite{Spi1}*{Theorem~4.8}]\label{thm:BS k-theory} 
For $c,d \geq 1$ and $S=BS(c,d)^+$, the $K$-theory of $\CQ(S)$ is given by
\[\begin{array}{lcccc}
K_0(\CQ(S)) &\cong& \Z/(d-1)\Z &\oplus& \delta_{1 c} \Z, \\
K_1(\CQ(S)) &\cong& \delta_{1 d}\Z &\oplus&  \Z/(c-1)\Z.
\end{array}\] 
\end{thm}

Note that $c=d=1$ gives $S = \N^2$. Let us focus on the case where $c,d \geq 2$. Recall from Example~\ref{ex:BS(c,d)^+ with cd > 1} that we have $S_{ci}^1 \cong \IF_d^+$. Thus $\CQ(S_{ci}^1) \cong \CO_d$ is simple and embeds into $\CQ(S)$. More importantly, the formulas from Theorem~\ref{thm:BS k-theory} suggest that this map is injective in $K$-theory, which is true by \cite{Spi1}, where it is shown that $[1]_0$ is the generator of $\Z/(d-1)\Z \subset K_0(\CQ(S))$. But this does not explain the parts of $K_*(\CQ(S))$ related to $c$. In that respect, we remark that the opposite semigroup $S^{\text{opp}}$ of $S$ coincides with $BS(d,c)^+$, and hence $(S^{\text{opp}})_{ci}^1 \cong \IF_c^+$. Another way of looking at $S^{\text{opp}}$ is to consider left ideals in $S$, and use the normal form $s=b^kw_n\cdots w_1$ with $w_i \in \{ab^j \mid 0 \leq j \leq c-1\}$, which has analogous properties to the normal form used in Example~\ref{ex:BS(c,d)^+ with cd > 1}. $S^{\text{opp}}$ plays the role of $S$ when considering right representations in place of left representations (bearing in mind that $C^*(S)$ is a universal model for the left regular representation of $S$ on $\ell^2(S)$). As $\CQ((S^{\text{opp}})_{ci}^1) \cong \CO_c$, the natural question arising from Theorem~\ref{thm:BS k-theory} is:

\begin{question}
Suppose $c,d >1$. Is there a $*$-homomorphism from the suspension of $\CO_c$ to $\CQ(BS(c,d)^+)$ that is injective in $K$-theory?
\end{question}

Let us close by pointing out that $\CQ(S)$ is a Kirchberg algebra if and only if $d$ does not divide $c$, see \cite{Spi1}*{Corollary~4.10}. This minimality condition was identified as the prerequisite for uniqueness of the $\kms_1$-state on $C^*(S)$ in Theorem~\ref{thm:KMS on BQD}~(ii).

\end{document}